\documentclass[11pt]{amsart}
\usepackage{graphicx}
\usepackage{float, verbatim}
\usepackage{enumerate}
\usepackage{color}
\usepackage{hyperref}
\usepackage{cite}
\usepackage{amsmath, amssymb, latexsym, bm, mathtools, braket, multirow, url, hhline}
\usepackage[top=30truemm, bottom=30truemm, left=20truemm, right=20truemm]{geometry}

\newcommand{\fint}[1]{\left\lfloor#1\right\rfloor}
\newcommand{\cint}[1]{\left\lceil#1\right\rceil}
\newcommand{\abs}[1]{\left|#1\right|}
\DeclareMathOperator{\id}{id}
\numberwithin{equation}{section}
\newcounter{count}
\newcommand{\num}{\stepcounter{count}\the\value{count}}

\renewcommand{\epsilon}{\varepsilon}

\newtheorem{theorem}{Theorem}[section]
\newtheorem{lemma}[theorem]{Lemma}
\newtheorem{corollary}[theorem]{Corollary}

\newtheorem{proposition}[theorem]{Proposition}

\theoremstyle{definition}
\newtheorem{question}[theorem]{Question}
\newtheorem{remark}[theorem]{Remark}

\begin{document}

\title[APs in the Graph of Slightly Curved Sequences]{Arithmetic Progressions in the Graphs of\\ Slightly Curved Sequences}

\author[K. Saito]{ Kota Saito }
\address{Kota Saito\\
Graduate School of Mathematics\\ Nagoya University\\ Furo-cho\\ Chikusa-ku\\ Nagoya\\ 464-8602\\ Japan }
\curraddr{}
\email{m17013b@math.nagoya-u.ac.jp}

\author[Y. Yoshida] {Yuuya Yoshida}
\address{Yuuya Yoshida\\
Graduate School of Mathematics\\ Nagoya University\\ Furo-cho\\ Chikusa-ku\\ Nagoya\\ 464-8602\\ Japan }
\curraddr{}
\email{m17043e@math.nagoya-u.ac.jp}

\thanks{KS is financially supported by Yoshida Scholarship Foundation.}
\thanks{YY is grateful to Seunghoan Song for providing many helpful comments for this paper.}


\maketitle

\begin{abstract}
A strictly increasing sequence of positive integers is called a \textit{slightly curved sequence with small error} 
if the sequence can be well-approximated by a function whose second derivative goes to zero faster than or equal to $1/x^\alpha$ for some $\alpha>0$. 
In this paper, we prove that arbitrarily long arithmetic progressions are contained in the graph of a slightly curved sequence with small error. 
Furthermore, we extend Szemer\'edi's theorem to a theorem about
slightly curved sequences. 
As a corollary, 
it follows that the graph of the sequence $\{\lfloor{n^a}\rfloor\}_{n\in A}$ contains arbitrarily long arithmetic progressions 
for every $1\le a<2$ and every $A\subset\mathbb{N}$ with positive upper 
density. 
Using this corollary, we show that 
the set $\set{\lfloor{\lfloor{p^{1/b}}\rfloor^a}\rfloor | \text{$p$ prime}}$ contains arbitrarily long arithmetic progressions for every $1\le a<2$ and $b>1$. 
We also prove that, for every $a\ge2$, the graph of $\{\lfloor{n^a}\rfloor\}_{n=1}^\infty$ does not contain any arithmetic progressions of length $3$.
\end{abstract}

\section{Introduction}\label{section1}

This paper considers problems involving arithmetic progressions. Let $k\geq 3$ and $d\ge1$ be integers. 
A sequence $\{a(j)\}_{j=0}^{k-1}\subset\mathbb{N}^d$ is called an \textit{arithmetic progression (AP) of length $k$} 
if there exists $D\in \mathbb{N}^d$ such that 
\[
	a(j)= a(0)+jD
\]
for all $j=0,1,\ldots, k-1$. We discuss only the cases $d=1$ and $d=2$. Here, note the following two points. 
First, all components of the common difference $D$ are positive, since $\mathbb{N}=\{1,2,\ldots\}$ does not contain zero. 
Second, this paper addresses strictly increasing sequences alone and then regards sequences of positive integers as subsets of $\mathbb{N}$. 
That is why we denote a sequence by, say, $\{a(n)\}_{n=1}^\infty$ instead of $(a(n))_{n=1}^\infty$. 

APs have elicited much interest from researchers studying arithmetic combinatorics, geometric measure theory, and fractal geometry. 
Most studies consider the density of sets to ensure the existence of long APs. For example, we recall Szemer\'edi's celebrated result: 

\begin{proposition}[Szemer\'edi \cite{szemeredi}]\label{szemeredi}
	For every $k\geq 3$ and $0<\delta\leq 1$ there exists an integer $N(k,\delta)>0$ such that if $N\geq N(k,\delta)$, 
	then every set $A\subset\{1,2,\ldots, N\}$ with $|A|\geq \delta N$ contains an AP of length $k$.
\end{proposition}

Here $|X|$ denotes the cardinality of a finite set $X$. 
Furthermore, Steinhaus showed that every set with positive Lebesgue measure contains arbitrarily long APs 
from Lebesgue's density theorem; for instance, see \cite[Theorem 3]{Jasinski}. 
These sufficient conditions are traditional and important, but it is difficult to weaken them. 
Hence we have found a new class of sets containing arbitrarily long APs, which is a main contribution of this paper. 
We call this new class \textit{slightly curved sequences}, which are defined below.

Let $g: \mathbb{N}\to\mathbb{R}$ be an eventually positive function and let $\mathbb{R}^+=(0,\infty)$. 
A strictly increasing sequence $\{a(n)\}_{n=1}^\infty\subset \mathbb{N}$ is called a \textit{slightly curved sequence with error $O(g(n))$} 
if there exists a twice differentiable function $f: \mathbb{R}^+ \to \mathbb{R}$ such that 
\begin{gather}
	f''(x)=O(1/x^\alpha), \label{SCS1}\\
	a(n)=f(n)+O(g(n))\nonumber
\end{gather}
for some $\alpha>0$. A slightly curved sequence with error $o(g(n))$ is also defined in the same way. 
Here the notation $f(x)=O(g(x))$ denotes that 
there exist $C>0$ and $x_0>0$ such that $|f(x)|\leq C g(x)$ for all $x\geq x_0$, 
where $g(x)$ is an eventually positive function. 
In this paper, the constant $C$ often depends on the length $k$ or the exponent $\alpha$ in \eqref{SCS1}. 
When emphasizing the dependence on $k$ or $\alpha$, 
we write $f(x)=O_{k,\alpha}(g(x))$. 
Also, the notation $f(x)=o(g(x))$ denotes that $f(x)/g(x)$ goes to zero as $x\to \infty$. 
Furthermore, to address two-dimensional APs, 
we define the \textit{graph} of a sequence: for every $A\subset\mathbb{N}$, 
the \textit{graph} of a sequence $\{a(n)\}_{n\in A}\subset \mathbb{N}$ is defined as the set $\{(n,a(n))\in\mathbb{N}^2 \mid n\in A\}$. 

Surprisingly, we do not impose any density conditions on sets of the above new class. 
Instead of density conditions, we restrict the behavior of leading terms of sequences. 
The first goal of this paper is to prove the following theorem: 

\begin{theorem}\label{main2}
	Fix an integer $k\geq 3$. The graph of every slightly curved sequence with error $o((\log\log n)^{1/c_k})$ contains an AP of length $k$, where $c_k=2^{2^{k+9}}$.
\end{theorem}

For small $k$, we can improve the above error term: 

\begin{theorem}\label{main2'}
	The graph of every slightly curved sequence with error $o((\log n)(\log\log n)^{-4})$ contains APs of length $3$. 
	The graph of every slightly curved sequence with error $o((\log n)^c)$ contains APs of length $4$, 
	where $c>0$ is an absolute constant.
\end{theorem}

Theorems~\ref{main2} and \ref{main2'} are proved in Section~\ref{prediscussion}. 
Since the second derivative \eqref{SCS1} of $f$ must vanish as $x\to\infty$, 
we cannot apply Theorem~\ref{main2} or \ref{main2'} to the case when $f$ is a quadratic function. 
To compensate this weak point, 
Appendix~\ref{appA} proves another result without the assumption $f''(x)=O(1/x^\alpha)$, 
which can be applied to the sequence $\{ \lfloor{c_2 n^2 + c_1 n + c_0}\rfloor \}_{n=1}^\infty$ with a sufficiently small $c_2>0$. 

To state the second goal of this paper, let us introduce basic notions. 
We say that a set $A\subset \mathbb{N}$ has \textit{positive upper density} if the condition 
\[
\limsup_{N\to \infty} \frac{|A\cap[1,N]|}{N} > 0
\]
holds. We also say that a set $A\subset \mathbb{N}$ has \textit{positive upper Banach density} if the condition 
\[
\limsup_{N\to \infty} \frac{\max_{n\ge0} |A\cap[n+1,n+N]|}{N} > 0
\]
holds. If a set $A\subset \mathbb{N}$ has positive upper density, then \textit{a fortiori} $A$ has also positive upper Banach density. 
However, the converse does not hold in general.

The second goal of this paper is to show the following result: 

\begin{theorem}[An extension of Szemer\'edi's theorem]\label{SZ}
	If $\{a(n)\}_{n=1}^\infty$ is a slightly curved sequence with error $O(1)$ and a set $A\subset\mathbb{N}$ has positive upper Banach density, 
	then the graph of $\{a(n)\}_{n\in A}$ contains arbitrarily long APs.
\end{theorem}

As a corollary, we also obtain the following result: 

\begin{corollary}\label{SZcor}
	If a set $A\subset \mathbb{N}$ has positive upper Banach density, then the graph of $\{\lfloor{n^a}\rfloor\}_{n\in A} $ contains arbitrarily long APs for every $1\leq a< 2$.
\end{corollary}

Here, for every $x\in \mathbb{R}$ the notation $\lfloor x \rfloor$ denotes the greatest integer less than or equal to $x$ 
and the notation $\lceil x \rceil$ denotes the least integer greater than or equal to $x$. 
Corollary~\ref{SZcor} with $a=1$ is just Szemer\'edi's theorem (Proposition~\ref{szemeredi}). 
By using Corollary~\ref{SZcor}, Section~\ref{Warming} shows that 
the set $\set{\lfloor{\lfloor{p^{1/b}}\rfloor^a}\rfloor | \text{$p$ prime}}$ contains arbitrarily long APs for every $1\le a<2$ and $b>1$.

\begin{proof}[Proof of Corollary~$\ref{SZcor}$ assuming Theorem~$\ref{SZ}$]
	Fix $1\leq a<2$ and let $f(x)=x^a$. Then $\lfloor n^a \rfloor = f(n) + O(1)$. 
	Since $f''(x) =O(1/x^{2-a})$, Theorem~\ref{SZ} implies Corollary~\ref{SZcor}.
\end{proof}

In particular, Corollary~\ref{SZcor} with $A=\mathbb{N}$ implies the following result immediately: 

\begin{corollary}\label{main3}
	The graph of $\{\lfloor n^a \rfloor\}_{n=1}^{\infty}$ contains arbitrarily long APs for every $1\leq a<2$.
\end{corollary}

When $a>1$ and $a\not\in\mathbb{N}$, the sequences $\{\lfloor n^a \rfloor\}_{n=1}^{\infty}$ are called Piatetski-Shapiro sequences. 
Hence Corollary~\ref{main3} implies that the Piatetski-Shapiro sequence with every exponent $1<a<2$ contains arbitrarily long APs.
Piatetski-Shapiro sequences are named after Piatetski-Shapiro \cite{Piatetski-Shapiro}, who proved that 
for every $1<a<12/11$ the sequence $\{\lfloor n^a \rfloor\}_{n=1}^{\infty}$ contains infinitely many primes. 
It is known that the range $1<a<12/11$ can be improved to $1<a<243/205$ \cite{Rivat-Wu}. 
Many preceding studies considered prime numbers contained in Piatetski-Shapiro sequences. 
For instance, Mirek \cite{Mirek} proved that for every $1<a<72/71$ 
the set of all primes of the form $\lfloor n^a \rfloor$ contains an AP of length $3$. 
Recently, Li and Pan \cite{Li-Pan} claimed that 
for every integer $k\ge3$ there exists $a_k>1$ such that for every $1<a<a_k$ 
the set of all primes of the form $\lfloor n^a \rfloor$ contains an AP of length $k$. 
Although this paper considers Piatetski-Shapiro sequences, 
prime numbers are not our main topic.

\begin{remark}
When $a>1$, the sum of the reciprocals of $\{\lfloor n^a \rfloor\}_{n=1}^{\infty}$ converges: 
\[
	\sum_{n=2}^\infty \frac{1}{\lfloor n^a \rfloor}
	\leq \sum_{n=2}^\infty \frac{1}{n^a-1}
	\leq \sum_{n=2}^\infty \frac{2}{n^a}< \infty.
\]
Thus we cannot apply the Erd\H{o}s-Tur\'an conjecture to this sequence. 
Here the Erd\H{o}s-Tur\'an conjecture asserts that 
every subset of positive integers whose sum of reciprocals diverges must contain arbitrarily long APs \cite{ErdosTuran}. 
Since the sum of the reciprocals of all primes diverges, 
the Erd\H{o}s-Tur\'an conjecture implies that the set of all primes contains arbitrarily long APs. 
Green and Tao \cite{GreenTao1} proved that the set of all primes contains arbitrarily long APs, 
but the Erd\H{o}s-Tur\'an conjecture is still open even in the case when the length of an AP is three. 
The Erd\H{o}s-Tur\'an conjecture is a strong statement that implies the result by Green and Tao; 
nevertheless, it does not imply Corollary~\ref{main3} directly.
\end{remark}

Corollary~\ref{main3} ensures that the graph of $\{\lfloor{n^a}\rfloor\}_{n=1}^\infty$ with $1\le a<2$ contains arbitrarily long APs, 
but the graph of $\{\lfloor{n^a}\rfloor\}_{n=1}^\infty$ with $a\ge2$ does not contain any APs of length $3$, 
which is proved in Section~\ref{FW}. 
Hence the graph of $\{n^2\}_{n=1}^{\infty}$ does not contain any APs of length $3$, 
while the sequence $\{n^2\}_{n=1}^{\infty}$ (not the graph) contains APs of length $3$, e.g., $\{1,25,49\}$. 
However, there are not any APs of length $4$ in $\{n^2\}_{n=1}^{\infty}$, 
which was proposed by Fermat in 1640 and first shown by Euler in 1780, according to Dickson's book \cite[pp.~440 and 635]{Dickson}. 
Euler also showed that there are not any APs of length $3$ in $\{n^3\}_{n=1}^{\infty}$, according to \cite[pp.~572--573]{Dickson}. 
Moreover, D\'enes \cite{Denes} showed that for every integer $3\le a\le30$ there are not any APs of length $3$ in $\{n^a\}_{n=1}^{\infty}$. 
Finally, for every integer $a\ge3$, Darmon and Merel \cite{DarmonMerel} showed that $\{n^a\}_{n=1}^{\infty}$ does not contain any APs of length $3$. 
We do not know whether $\{\lfloor n^a \rfloor\}_{n=1}^{\infty}$ would contain long APs if $a>2$ is not an integer. 
One might guess that such a sequence would not contain APs, but the following sequences are APs of length $4$: 
\begin{gather*}
	\{ \lfloor{2^{2.2}}\rfloor, \lfloor{11^{2.2}}\rfloor, \lfloor{15^{2.2}}\rfloor, \lfloor{18^{2.2}}\rfloor \},\\
	\{ \lfloor{14^{2.655015}}\rfloor, \lfloor{39^{2.655015}}\rfloor, \lfloor{50^{2.655015}}\rfloor, \lfloor{58^{2.655015}}\rfloor \},\\
	\{ \lfloor{27^{2.720398}}\rfloor, \lfloor{89^{2.720398}}\rfloor, \lfloor{114^{2.720398}}\rfloor, \lfloor{132^{2.720398}}\rfloor \}.
\end{gather*}

In the course of this study, we proved a result in Appendix~\ref{appB}, 
which is that the graph of $\{\lfloor{f(n)}\rfloor\}_{n=1}^\infty$ contains an AP of length 4 
if a function $f:\mathbb{N}\to\mathbb{R}^+$ satisfies $\Delta f>0$, $\Delta^2 f\geq 0$, and $\liminf_{n\to\infty} f(n)/n^2 < 1/18$, 
where $\Delta$ denotes the difference operator which is defined in Section~\ref{prediscussion}. 
This result can be applied to the sequence $\{ \lfloor{c_2 n^2 + c_1 n + c_0}\rfloor \}_{n=1}^\infty$ with $0<c_2<1/18$, 
but Theorem~\ref{main2} or \ref{main2'} cannot be applied.

\section{Warming up}\label{Warming}
As an application of Corollary~\ref{SZcor}, this section shows that 
the set $\set{\lfloor{\lfloor{p^{1/b}}\rfloor^a}\rfloor | \text{$p$ prime}}$ contains arbitrarily long APs for every $1\le a<2$ and $b>1$. 
Once we put $A=\set{\lfloor{p^{1/b}}\rfloor | \text{$p$ prime}}$, 
the above set can be expressed as 
$\set{\lfloor{\lfloor{p^{1/b}}\rfloor^a}\rfloor | \text{$p$ prime}} = \{\lfloor{n^a}\rfloor\}_{n\in A}$. 
Thus, if $A$ has positive upper density, the above statement follows from Corollary~\ref{SZcor}. 
Let us verify that $A$ has positive upper density. 
Let $\pi(x)$ be the number of primes less than or equal to $x>0$ 
and $\tilde{\pi}(S)$ be the number of primes contained in a set $S\subset\mathbb{R}^+$. 
Then every $x>0$ satisfies 
\begin{align*}
	&\quad \abs{A\cap[1,x]} = \abs{\set{\lfloor{p^{1/b}}\rfloor \le x | \text{$p$ prime}}}\\
	&\ge \frac{\abs{\set{p\le x^b | \text{$p$ prime}}}}{\max_{n\le x} \tilde{\pi}[n^b, (n+1)^b)}
	\ge \frac{\pi(x^b)}{\max_{n\le x} \tilde{\pi}[n^b, (n+1)^b]}.
\end{align*}
Since there exists $y_0>0$ such that every $x>0$ and $y\ge y_0$ satisfy $\tilde{\pi}[x, x+y]\le3y/\log y$ \cite[Corollary~3.4]{MV}, 
the denominator $\tilde{\pi}[n^b, (n+1)^b]$ of the above right-hand side is upper bounded as 
\begin{align*}
	&\quad \tilde{\pi}[n^b, (n+1)^b]
	\overset{\text{(i)}}{\le} \max_{n\le x} \tilde{\pi}[n^b, n^b + b(n+1)^{b-1}]\\
	&\le \max_{n\le x} \tilde{\pi}[n^b, n^b + b(x+1)^{b-1}]
	\overset{\text{(ii)}}{\le} \frac{3b(x+1)^{b-1}}{\log(b(x+1)^{b-1})}
\end{align*}
for every $x\ge(y_0/b)^{1/(b-1)}-1$, 
where (i) and (ii) follow from the mean value theorem and \cite[Corollary~3.4]{MV}, respectively. 
We now use the prime number theorem \cite[Chapter~6]{MV}: $\pi(x)\sim x/\log x$, 
where the notation $f(x)\sim g(x)$ denotes that $f(x)/g(x)$ goes to one as $x\to\infty$. 
Therefore, the prime number theorem yields 
\begin{align*}
	&\quad \frac{\abs{A\cap[1,x]}}{x} \ge \pi(x^b) \Big/ x\frac{3b(x+1)^{b-1}}{\log(b(x+1)^{b-1})}\\
	&\sim \frac{x^b}{b\log x} \Big/ x\frac{3b(x+1)^{b-1}}{(b-1)\log(x+1)}
	\sim \frac{b-1}{3b^2} > 0 \quad (x\to\infty),
\end{align*}
whence $A$ has positive upper density.

As a special case of the above result, it follows that 
the set $\set{\lfloor{\lfloor{p^{1/b}}\rfloor^b}\rfloor | \text{$p$ prime}}$ contains arbitrarily long APs for every $1<b<2$. 
However, the case $b=1$, i.e., the Green-Tao theorem \cite{GreenTao1} would not be proved by the above method. 
This is because the set $\set{\lfloor{p^{1/b}}\rfloor | \text{$p$ prime},\ |\lfloor{p^{1/b}}\rfloor^b - p| \le C}$ 
has probably upper density zero for every $1<b<2$ and $C>0$.

\section{More general statement than Theorems~\ref{main2} and \ref{main2'}}\label{prediscussion}
In order to prove Theorems~\ref{main2} and \ref{main2'}, 
this section introduces another proposition which is proved in Section~\ref{proofTh1}. 
This proposition uses a function satisfying three properties below. 
Before stating them, let us define the \textit{van der Waerden numbers} \cite{Waerden}. 
For every $k\geq 3$ and $r\geq 2$, 
the \textit{van der Waerden number} $W(r,k)$ is the smallest number $N$ such that 
if $\{1,2,\ldots,N\}$ are partitioned into $r$ different sets then there exists at least one set which contains an AP of length $k$. 
Fix an integer $k\geq 3$. 
This section considers a function $U_k$ satisfying the following properties:

\begin{itemize}
	\item[(U1)]
	there exist $x_k\geq 1$ and $y_k\geq 1$ such that a function 
	$U_k : [x_k, \infty) \to [y_k,\infty)$ is increasing and bijective;
	\item[(U2)]
	$U_k(r^\alpha) = O_{k,\alpha}(U_k(r))$ for all $\alpha>0$;
	\item[(U3)]
	the inverse function $U_k^{-1}$ of $U_k$ satisfies $W(r,k)=O_k(U_k^{-1}(r))$.
\end{itemize}

Using such a function $U_k$, we obtain the following proposition: 

\begin{proposition}\label{main1}
	Fix an integer $k\geq 3$. Assume that a function $U_k$ satisfies (U1), (U2), and (U3). 
	If $\{a(n)\}_{n=1}^\infty$ is a slightly curved sequence with error $o(U_k(n))$, 
	then there exists an arithmetic progression $P$ of length $k$ such that $a(P)$ is also an arithmetic progression of length $k$.
\end{proposition}

The proof of Theorem~\ref{main2} uses the function $U_k(r)=(\log\log r)^{1/c_k}$ with $c_k=2^{2^{k+9}}$. 
Similarly, the proof of Theorem~\ref{main2'} uses the functions 
$U_3(r)=C_3^{-1}(\log r)(\log\log r)^{-4}$ and $U_4(r)=C_4^{-1}(\log r)^c$ for some $C_3,C_4>0$ and some absolute constant $c>0$. 
Thus Theorems~\ref{main2} and \ref{main2'} are special cases of Proposition~\ref{main1}. 
To prove Theorem~\ref{main2} by assuming Proposition~\ref{main1}, 
we need to use upper bounds of the van der Waerden numbers. 
As such upper bounds, Gowers' upper bounds \cite[Theorem~18.6]{Gowers} are known: 
\begin{equation}
	W(r,k)\leq 2^{2^{r^{c_k}}},\quad c_k=2^{2^{k+9}}. \label{Gowers}
\end{equation}
For general $r$ and $k$, Gowers' upper bounds are best at present.

\begin{proof}[Proof of Theorem~$\ref{main2}$ assuming Proposition~$\ref{main1}$]
Fix an integer $k\geq 3$. Let $U_k(r) = (\log\log r)^{1/c_k}$ with $c_k=2^{2^{k+9}}$. 
It is enough to show that the function $U_k$ satisfies (U1), (U2), and (U3). 
The properties (U1) and (U2) follow from the definition of $U_k$. 
Also, Gowers' upper bound \eqref{Gowers} implies that the function $U_k$ satisfies $W(r,k)\leq \exp\bigl(\exp(r^{c_k})\bigr)=U_k^{-1}(r)$, which is just (U3). 
\end{proof}

Next, to prove Theorem~\ref{main2'} by assuming Proposition~\ref{main1}, 
we introduce the notation $r_k(N)$ \cite{ErdosTuran}. 
The notation $r_k(N)$ is the maximum cardinality of all subsets of $\{1,2,\ldots,N\}$ that contain no APs of length $k$. 
For $k=3,4$, the following upper bounds of $r_k(N)$ are known \cite[Theorem~1.1]{Bloom}, \cite[Theorem~1.1]{GreenTao2}: 
\begin{equation}
	r_3(N) = O\Bigl( \frac{N}{(\log N)(\log\log N)^{-4}} \Bigr),\quad
	r_4(N) = O\Bigl( \frac{N}{(\log N)^c} \Bigr) \label{rUB}
\end{equation}
for some absolute constant $c>0$. 
As explained below, 
the upper bounds \eqref{rUB} correspond to upper bounds of the van der Waerden numbers $W(r,k)$ with $k=3,4$, 
which are better than Gowers' upper bounds \eqref{Gowers}.

\begin{proof}[Proof of Theorem~$\ref{main2'}$ assuming Proposition~$\ref{main1}$]
First, we show that if a function $U_k$ satisfies (U1) and 
\begin{equation}
	r_k(N) < N/U_k(N)\quad(N\ge x_k) \label{eqY1}
\end{equation}
then the function $U_k$ satisfies (U3). 
Due to (U1), there exists $r_k>0$ such that every $r\ge r_k$ satisfies $U_k^{-1}(r)\ge1$. 
Let $r\ge r_k$ be an integer and let $N=\lceil U_k^{-1}(r) \rceil$. 
The definition of $N$ yields $N = \lceil{U_k^{-1}(r)}\rceil \ge U_k^{-1}(r)$ 
and then (U1) does $U_k(N)\ge r$. 
Take an arbitrary partition of $\{1,2\ldots, N\}$ into $r$ small sets. 
Then the pigeonhole principle implies that there exists a set $A$ consisting of at least $\lceil N/r\rceil$ elements. 
Thus $A$ contains an AP of length $k$ because the inequality 
\[
	|A| \geq \lceil N/r\rceil \geq N/r \overset{\text{(i)}}{\geq} N/U_k(N) \overset{\eqref{eqY1}}{>} r_k(N)
\]
holds, where (i) follows from $U_k(N)\ge r$. 
Since the partition of $\{1,2\ldots, N\}$ is arbitrary, 
the van der Waerden number $W(r,k)$ is upper bounded by $N$. 
Thus it follows that 
\[
	W(r,k) \leq N \leq U_k^{-1}(r)+1 \leq 2U_k^{-1}(r)
\]
from the definition of $N$ and (U1). 
Therefore, the function $U_k$ satisfies (U3).
\par
Next, using the result in the previous paragraph, we show Theorem~\ref{main2'}. 
The upper bounds \eqref{rUB} imply that there exist constants $C_3,C_4>0$ such that 
\[
r_3(N) < C_3\cdot \frac{N}{(\log N)(\log\log N)^{-4}},\quad
r_4(N) < C_4\cdot \frac{N}{(\log N)^c}
\]
for every sufficiently large $N$. 
If taking the functions $U_3(N)=C_3^{-1}(\log N)(\log\log N)^{-4}$ and $U_4(N)=C_4^{-1}(\log N)^c$, 
the inequality \eqref{eqY1} holds for $k=3$ and $k=4$, respectively. 
Since the functions $U_3$ and $U_4$ satisfy (U1) and (U2), 
the result in the previous paragraph implies that the functions $U_3$ and $U_4$ also satisfy (U3). 
Therefore, Proposition~\ref{main1} implies Theorem~\ref{main2'}.
\end{proof}

\section{Proof of Proposition~\ref{main1}}\label{proofTh1}
This section proves Proposition~\ref{main1}. 
Before proving Proposition~\ref{main1}, 
let us define a semi-norm on the vector space $\mathcal{F} = \{f \mid f:\mathbb{R}^+\to \mathbb{R}\}$. 
Let $k\geq 3$ be an integer and $P=\{b(j)\}_{j=0}^{k-1} \subset \mathbb{R}^+$ be a strictly increasing sequence. 
We define
\[
	N_P(f)=\sum_{j=0}^{k-3} \abs{\Delta^2[f\circ b](j)},
\]
for every $f\in\mathcal{F}$, where $\Delta$ denotes the difference operator, that is,
\[
	\Delta f(x)=f(x+1)-f(x),
\] 
and $\Delta^2 := \Delta\circ\Delta$. We can find that $N_P$ satisfies the following properties:
\begin{itemize}
	\item[(N1)]
	for every strictly increasing function $f\in\mathcal{F}$, 
	\begin{equation*}
		\text{$N_P(f)=0$ if and only if $f(P)$ is an AP of length $k$};
	\end{equation*}
	\item[(N2)]
	$N_P(f)\geq 0$ for all $f\in\mathcal{F}$;
	\item[(N3)]
	$N_P(f +g ) \leq N_P(f) + N_P(g)$ for all $f,g\in\mathcal{F}$.
\end{itemize}
All the properties above can be easily checked from the definition.

\begin{proof}[Proof of Proposition~$\ref{main1}$]
Fix an integer $k\geq 3$ and take an arbitrary slightly curved sequence $\{a(n)\}_{n=1}^\infty$ with error $o(U_k(n))$: 
there exists a twice differentiable function $f: \mathbb{R}^+ \to \mathbb{R}$ satisfying \eqref{SCS1} and $a(n)=f(n)+o(U_k(n))$. 
Let $r$ be a sufficiently large positive integer. We define $W(r) = W(r,k)$, $R(n) = a(n)-f(n)$, and 
\[
A(r) = \Set{m\in\mathbb{N} | \lfloor W(r)^{3/\alpha} \rfloor \leq m < \lfloor W(r)^{3/\alpha}\rfloor+W(r)}.
\]
Then every $m\in A(r)$ satisfies 
\begin{equation*}
	R(m) = o\bigl(U_k(2W(r)^{\max\{3/\alpha,1\}}) \bigr) = o\bigl(U_k(U_k^{-1}(r))\bigr)=o(r)
\end{equation*}
thanks to (U1), (U2), and (U3). Thus there exists a positive function $\delta(r)$ such that 
\begin{equation*}
	R(m)\in [-\delta(r), \delta(r))
\end{equation*}
for all $m\in A(r)$, and
\begin{equation}\label{f2.5}
\delta(r)=o(r)
\end{equation}
as $r\to \infty$. We also define
\begin{gather*}
	I_j = \Bigl[-\delta(r) + \frac{2\delta(r)}{r}j,\ -\delta(r) + \frac{2\delta(r)}{r}(j+1)\Bigr),\quad
	A_j(r) = \{m\in A(r) \mid R(m)\in I_j \}
\end{gather*}
for all $j=0,1,\ldots, r-1$. 
Note that the union of all the small sets $A_j(r)$ equals $A(r)$. 
From the definition $W(r)=W(r,k)$, it follows that there exists an integer $q\in \{0,1,\ldots, r-1\}$ such that 
$A_q(r)$ contains an arithmetic progression $P=\{b(j)\}_{j=0}^{k-1}$ of length $k$. 
Here $b(j)$ is expressed as $b(j) = dj + e$ with two integers $d,e>0$. 
Since every $j=0,1,\ldots ,k-1$ satisfies $b(j)\in A(r)$, we obtain 
\begin{equation}\label{ff3}
	d\leq W(r),\quad \lfloor W^{3/\alpha}(r) \rfloor\leq e.
\end{equation}
Then the triangle inequality (N3) implies $N_P(a) \leq N_P(f)+N_P(R)$.
\par
Finally, we show that $N_P(a)=o(1)$ as $r\to\infty$, which follows from the two relations $N_P(f)=o(1)$ and $N_P(R)=o(1)$. 
Since we have \eqref{f2.5} and $b(j)\in A_q(r)$ for all $j=0,1,\ldots, k-1$, the inequality 
\[
N_P(R) \le \sum_{j=0}^{k-3} (|\Delta[R\circ b](j+1)| + |\Delta[R\circ b](j)|)
\leq \frac{4(k-2)\delta(r)}{r} = o(1)
\]
holds as $r\to \infty$. 
Hence the remaining is to show that $N_P(f)=o(1)$ as $r\to \infty$. 
The mean value theorem implies that for every $j=0,1,\ldots, k-3$ there exist $\theta_j, \eta_j \in (0,1)$ such that
\begin{align*}
	\Delta^2[f\circ b](j)
	&= \Delta[f\circ b](j+1) - \Delta[f\circ b](j)
	= (\Delta[f\circ b])'(j+\theta_j) \\
	&= d\{f'\circ b(j+\theta_j+1) - f'\circ b(j+\theta_j)\}
	= d^2 f''\circ b(j+\theta_j+\eta_j),
\end{align*}
where $b(x) := dx+e$. 
Since \eqref{ff3} and the assumption $f''(x)=O( 1/x^\alpha)$ hold, every $j=0,1,\ldots, k-3$ satisfies 
\[
	\Delta^2[f\circ b](j)=d^2 f''(d(j+\theta_j+\eta_j)+e) = O\Bigl(W(r)^2 \cdot \frac{1}{W(r)^3}\Bigr) =o(1)
\]
as $r\to \infty$, whence $N_P(f)=o(1)$ as $r\to \infty$. 
Therefore, a sufficiently large integer $r>0$ satisfies $N_P(a)<1$. 
Since $N_P(a)$ is a non-negative integer from the definition, the equation $N_P(a)=0$ holds, 
which implies that $a(P)$ is also an AP of length $k$ due to (N1).
\end{proof}

\section{Proof of Theorem~\ref{SZ}}
This section considers, for every $\alpha>0$, whether the set $A\cap(h_\alpha(x),h_\alpha(x+1)]$ contains an AP or not, 
which is useful to prove Theorem~\ref{SZ}. 
Here $h_\alpha$ is a function satisfying the following properties: 
\begin{itemize}
	\item[(H1)]
	$h_\alpha$ is a differentiable function defined on the interval $(x_0,\infty)$ for some $x_0>0$; 
	\item[(H2)]
	there exists a positive number $x_1\ge x_0$ such that $h_\alpha$ and $h'_\alpha$ strictly increase on the interval $(x_1,\infty)$; 
	\item[(H3)]
	$\lim_{x\to\infty} h'_\alpha(x+1)^2/h_\alpha(x)^\alpha = 0$;
	\item[(H4)]
	$\lim_{x\to\infty} h'_\alpha(x)=\infty$.
\end{itemize}
For example, the function $h_\alpha(x)=x\log x$ satisfies the above three properties. 
As another example, we can take the function $h_\alpha(x)=x^{1+\alpha/2}$. 
The latter example is used in Appendix~\ref{appA}. 
In order to prove Theorem~\ref{SZ}, we introduce the following condition for a set $A\subset\mathbb{N}$: 
\begin{itemize}
	\item[(C)]
	for every integer $k\ge3$ and every $\alpha>0$, there exist a function $h_\alpha$ satisfying (H1)--(H3) and 
	a strictly increasing sequence $\{x_n\}_{n=1}^\infty\subset\mathbb{R}^+$ which diverges such that 
	for every integer $n>0$ the set $A\cap(h_\alpha(x_n),h_\alpha(x_n+1)]$ contains an AP of length $k$.
\end{itemize}

Theorem~\ref{SZ} follows from the following proposition and preliminary lemma.

\begin{proposition}\label{main5}
	Let $A$ be a subset of positive integers with (C), and $\{a(n)\}_{n=1}^\infty$ be a slightly curved sequence with error $O(1)$. 
	Then the graph of $\{a(n)\}_{n\in A}$ contains arbitrarily long APs.
\end{proposition}
\begin{proof}
	Let $k\ge3$ be an integer. 
	There exists a twice differentiable function $f: \mathbb{R}^+\to\mathbb{R}$ satisfying \eqref{SCS1} and $a(n)=f(n)+O(1)$. 
	Defining $R(n) = a(n) - f(n)$, we can take an integer $M$ satisfying $|R(n)|<M$. 
	Put $r = 4kM$. 
	Then (C) implies that there exist a function $h_\alpha$ satisfying (H1)--(H3) and 
	a strictly increasing sequence $\{x_n\}_{n=1}^\infty\subset\mathbb{R}^+$ which diverges such that 
	for every integer $n>0$ the set $A\cap(h_\alpha(x_n),h_\alpha(x_n+1)]$ contains an AP of length $W(r,k)$. 
	Let $P(n)$ be such an AP of length $W(r,k)$. 
	Then we define the following sets: 
	\begin{gather*}
		I_j=\Bigl[ -M + \frac{2M}{r}j, -M + \frac{2M}{r}(j+1) \Bigr)\quad(0\le j\le r-1),\\
		A_j(n)=\set{m\in P(n) | R(m)\in I_j}\quad(0\le j\le r-1).
	\end{gather*}
	Since the union of all the small sets $A_j(n)$ equals $P(n)$, 
	for some $0\le q\le r-1$ the set $A_q(n)$ contains an arithmetic progression $\tilde{P}=\{b(j)\}_{j=0}^{k-1}$. 
	Here $b(j)$ is expressed as $b(j) = dj + e$ with two integers $d,e>0$. 
	For every sufficiently large $n$ the inequality $d \le h_\alpha(x_n+1) - h_\alpha(x_n) \le h'_\alpha(x_n+1)$ holds due to (H2), and 
	the inequality $e\ge h_\alpha(x_n)$ also holds. 
	Thus every sufficiently large $n$ satisfies 
	\begin{align}
	\begin{split}
		N_{\tilde{P}}(f) &= \sum_{j=0}^{k-3} \abs{\Delta^2[f\circ b](j)}
		= \sum_{j=0}^{k-3} d^2\abs{f''(d(j + \theta_j + \eta_j) + e)}\\
		&= O\Bigl( \frac{h'_\alpha(x_n+1)^2}{h_\alpha(x_n)^\alpha} \Bigr) = o(1),
	\end{split}\label{eqY3}
	\end{align}
	where the last equality follows from (H3) 
	and the values $\theta_j=\theta_j(d,e)$ and $\eta_j=\eta_j(d,e)$ are real numbers satisfying $\theta_j,\eta_j\in(0,1)$. 
	Moreover, the relation $\tilde{P}\subset A_q(n)$ implies 
	\begin{align}
	\begin{split}
		N_{\tilde{P}}(R) &= \sum_{j=0}^{k-3} \abs{\Delta^2[R\circ b](j)}
		\le \sum_{j=0}^{k-3} (\abs{\Delta[R\circ b](j+1)} + \abs{\Delta[R\circ b](j)})\\
		&\le 2(k-2)\frac{2M}{r} = 1-2/k.
	\end{split}\label{eqY4}
	\end{align}
	Hence a sufficiently large integer $n$ satisfies $N_{\tilde{P}}(f) < 2/k$. 
	The inequality $N_{\tilde{P}}(a)<1$ follows from \eqref{eqY3} and \eqref{eqY4}. 
	Since $N_{\tilde{P}}(a)$ is a non-negative integer, 
	the equation $N_{\tilde{P}}(a)=0$ follows: $a(\tilde{P})$ is an AP of length $k$. 
	Therefore, the graph of $\{a(n)\}_{n\in A}$ contains an AP of length $k$.
\end{proof}

\begin{lemma}\label{SZlma1}
	If a set $A\subset\mathbb{N}$ has positive upper Banach density, 
	then there exists a function $h=h_\alpha$ (to be independent of $\alpha$) satisfying (H1)--(H4), and 
	\[
	\limsup_{x\to \infty} \frac{|A\cap(h(x),h(x+1)]|}{|\mathbb{N}\cap(h(x),h(x+1)]|} > 0.
	\]
\end{lemma}
\begin{proof}
	\setcounter{count}{0}
	\noindent\textbf{Step~\num.}
	For each integer $N>0$, we take a non-negative integer $F(N)$ satisfying 
	\[
	\max_{n\ge0} |A\cap[n+1,n+N]| = |A\cap[F(N)+1,F(N)+N]|
	\]
	and then define the functions $F_{\max}: [0,\infty)\to[0,\infty)$ and $\tilde{F}: [0,\infty)\to[0,\infty)$ as follows: 
	\begin{align*}
		F_{\max}(x) &= (N+1-x)\max_{1\le n\le N} F(n) + (x-N)\max_{1\le n\le N+1} F(n)\\
		&\qquad(N\in\mathbb{N}\cup\{0\},\ N\le x<N+1),\\
		\tilde{F}(x) &= x e^x (F_{\max}(x)+1),
	\end{align*}
	where $\max_{1\le n\le x} F(n)=0$ for all $0\le x<1$. 
	Note that $F_{\max}$ is continuous and satisfies $F_{\max}(N)=\max_{1\le n\le N} F(n)$. 
	Thus $\tilde{F}$ is continuous and satisfies the following properties: 
	\begin{enumerate}
		\item\label{p1}
		$F(N)+1 \le \tilde{F}(N)$ for all $N\in\mathbb{N}$; 
		\item\label{p2}
		$\tilde{F}(x) \ge e^x$ for all $x\ge1$; 
		\item\label{p3}
		$\tilde{F}(2x) \ge 2\tilde{F}(x)$ for all $x\ge0$; 
		\item\label{p4}
		$\tilde{F}: [0,\infty)\to[0,\infty)$ is a homeomorphism and 
		satisfies $\tilde{F}(0)=0$ and $\lim_{x\to\infty} \tilde{F}(x)=\infty$.
	\end{enumerate}
	Using this function, we define the function $h: [0,\infty)\to[0,\infty)$ as 
	\[
	h(x) = \int_0^x \tilde{F}^{-2}(y)\,dy,
	\]
	where $\tilde{F}^{-2} := \tilde{F}^{-1}\circ\tilde{F}^{-1}$, and $\tilde{F}^{-1}$ denotes the inverse function of $\tilde{F}$. 
	Since $\tilde{F}^{-2}$ is continuous, the function $h$ is continuously differentiable. 
	Also, $h$ satisfies (H1) and (H2).
	\par
	\noindent\textbf{Step~\num.}
	To prove that $h$ satisfies (H3), we show the following facts: 
	\begin{enumerate}
		\item\label{f1}
		$h(2x) \le 4h(x)$ for all $x\ge0$; 
		\item\label{f2}
		$h(x) \le x\log x$ for every sufficiently large $x$; 
		\item\label{f3}
		$x h'(x) \le 4h(x)$ for all $x\ge0$; 
		\item\label{f4}
		$h^{-1}(x) \le x$ for all $x\ge\tilde{F}^2(2)$,
	\end{enumerate}
	where $\tilde{F}^2 := \tilde{F}\circ\tilde{F}$. 
	First, thanks to property~\ref{p3} in Step~1, every $y\ge0$ satisfies $2\tilde{F}^{-1}(y) \ge \tilde{F}^{-1}(2y)$. 
	Therefore, every $x\ge0$ satisfies 
	\[
	h(2x) = \int_0^{2x} \tilde{F}^{-2}(y)\,dy = 2\int_0^x \tilde{F}^{-2}(2y)\,dy
	\le 4\int_0^x \tilde{F}^{-2}(y)\,dy = 4h(x),
	\]
	which is just fact~\ref{f1}. 
	Second, thanks to property~\ref{p2} in Step~1, every $y\ge\tilde{F}(1)$ satisfies $\tilde{F}^{-1}(y) \le \log y$. 
	Therefore, every sufficiently large $x$ satisfies 
	\[
	h(x) = \int_{\tilde{F}^2(1)}^x \tilde{F}^{-2}(y)\,dy + O(1)
	\le \int_{\tilde{F}^2(1)}^x \log\log y\,dy + O(1)
	\le x\log\log x \le x\log x,
	\]
	which is just fact~\ref{f2}. 
	Third, fact~\ref{f3} is verified as follows: 
	\[
	x h'(x) \le \int_x^{2x} h'(y)\,dy \le h(2x) \le 4h(x),
	\]
	where the last inequality follows from fact~\ref{f1}. 
	Fourth, every $x\ge\tilde{F}^2(2)$ satisfies 
	\[
	h(x) \ge \int_{\tilde{F}^2(2)}^x \tilde{F}^{-2}(y)\,dy
	\ge \int_{\tilde{F}^2(2)}^x 2\,dy
	= 2(x-\tilde{F}^2(2)) \ge x,
	\]
	which implies fact~\ref{f4}.
	\par
	\noindent\textbf{Step~\num.}
	We show that $h$ satisfies (H3) for every $\alpha>0$. 
	Facts~\ref{f1} and \ref{f3} and (H2) imply 
	\begin{equation}
		\frac{h'(x+1)^2}{h(x)^\alpha}
		\le \frac{h'(2x)^2}{h(x)^\alpha}
		\le \Bigl( \frac{4h(2x)}{2x} \Bigr)^2 \frac{1}{h(x)^\alpha}
		\le \Bigl( \frac{8h(x)}{x} \Bigr)^2 \frac{1}{h(x)^\alpha}
		= \frac{8^2 h(x)^{2-\alpha}}{x^2}. \label{eq1}
	\end{equation}
	If $\alpha\ge2$, the right-hand side in \eqref{eq1} goes to zero as $x\to\infty$. 
	If $0<\alpha<2$, fact~\ref{f2} in Step~2 implies 
	\[
	\frac{h(x)^{2-\alpha}}{x^2}
	\le \frac{(x\log x)^{2-\alpha}}{x^2}
	= \frac{(\log x)^{2-\alpha}}{x^\alpha} \xrightarrow{x\to\infty} 0.
	\]
	Therefore, the function $h$ satisfies (H3).
	\par
	\noindent\textbf{Step~\num.}
	Assuming 
	\[
	\limsup_{x\to \infty} \frac{|A\cap(h(x),h(x+1)]|}{|\mathbb{N}\cap(h(x),h(x+1)]|} = 0,
	\]
	we now deduce a contradiction. This assumption implies that 
	for every $0<\epsilon<1/4$ there exists $x_0>0$ such that every $x\geq x_0$ satisfies 
	\[
	|A\cap(h(x),h(x+1)]| \leq \epsilon |\mathbb{N}\cap(h(x),h(x+1)]|.
	\]
	We consider only the case when $N$ is sufficiently large below. 
	If $F(N)+1 \le h(x_0)$, letting $l = \lceil h^{-1}(F(N)+N) - x_0 \rceil$, we have 
	\begin{align}
		&\quad |A\cap[F(N)+1,F(N)+N]|\nonumber\\
		&\leq |A\cap[F(N)+1,h(x_0)]| + \sum_{j=1}^l |A\cap(h(x_0+j-1),h(x_0+j)]|\nonumber\\
		&\leq h(x_0) - F(N) + \epsilon|\mathbb{N}\cap(h(x_0),h(x_0+l)]|
		\leq \epsilon h(x_0+l) - F(N) + O(1)\nonumber\\
		&\leq \epsilon h\bigl( h^{-1}(F(N)+N)+1 \bigr) - F(N) + O(1)
		\leq \epsilon h\bigl( 2h^{-1}(F(N)+N) \bigr) - F(N) + O(1)\nonumber\\
		&\overset{\text{(i)}}{\leq} 4\epsilon(F(N)+N) - F(N) + O(1)
		= 4\epsilon N - (1-4\epsilon)F(N) + O(1)
		\leq 4\epsilon N + O(1), \label{eq3}
	\end{align}
	where (i) follows from fact~\ref{f1} in Step~2. 
	If $F(N)+1 > h(x_0)$, 
	taking a positive number $x_1$ and a positive integer $l$ with $h(x_1) = F(N)+N$ and $h(x_1-l) < F(N)+1 \le h(x_1-l+1)$, 
	we have 
	\begin{align}
		&\quad |A\cap[F(N)+1,F(N)+N]| \le |A\cap(h(x_1-l),h(x_1)]|\nonumber\\
		&\leq \sum_{j=1}^l |A\cap(h(x_1-j),h(x_1-j+1)]|\nonumber\\
		&\leq \epsilon|\mathbb{N}\cap(h(x_1-l),h(x_1)]|
		\leq \epsilon\{ h(x_1) - h(x_1-l) \}\nonumber\\
		&\leq \epsilon\bigl\{ F(N)+N - h\bigl( h^{-1}(F(N)+1) - 1 \bigr) \bigr\}\nonumber\\
		&\overset{\text{(ii)}}{\leq} \epsilon\bigl\{ F(N)+N - (F(N)+1) + h'\bigl( h^{-1}(F(N)+1) \bigr) \bigr\}\nonumber\\
		&\leq \epsilon\bigl\{ N + h'\bigl( h^{-1}(F(N)+1) \bigr) \bigr\}
		\overset{\text{(iii)}}{=} \epsilon N + o(N). \label{eq4}
	\end{align}
	\par
	We verify the above (ii) and (iii). 
	First, the mean value theorem and (H2) imply 
	\[
	h(x-1) \ge h(x)-h'(x)
	\]
	for every $x\ge1$. Putting $x=h^{-1}(F(N)+1)$, we obtain the inequality (ii). 
	Next, in order to prove (iii), we show 
	\begin{equation}
		\lim_{x\to\infty} \frac{(h'\circ h^{-1}\circ\tilde{F}\circ h')(x)}{h'(x)} = 0. \label{eq2}
	\end{equation}
	Using $h'(x)=\tilde{F}^{-2}(x)$, we obtain 
	\begin{align*}
		&\quad \frac{(h'\circ h^{-1}\circ\tilde{F}\circ h')(x)}{h'(x)}
		= \frac{(h'\circ h^{-1}\circ\tilde{F}^{-1})(x)}{h'(x)}\\
		&\overset{\text{(iv)}}{\le} \frac{(h'\circ\tilde{F}^{-1})(x)}{h'(x)}
		= \frac{\tilde{F}^{-3}(x)}{\tilde{F}^{-2}(x)}
		\overset{\text{(v)}}{\le} \frac{\log\tilde{F}^{-2}(x)}{\tilde{F}^{-2}(x)}
		\xrightarrow{x\to\infty} 0,
	\end{align*}
	where (iv) and (v) follow from fact~\ref{f4} in Step~2 and property~\ref{p2} in Step~1, respectively. 
	From property~\ref{p1} in Step~1, (H2), and \eqref{eq2}, it follows immediately that 
	\[
	\frac{h'\bigl( h^{-1}(F(N)+1) \bigr)}{N} \le \frac{h'\bigl( h^{-1}(\tilde{F}(N)) \bigr)}{N} \xrightarrow{N\to\infty} 0,
	\]
	which is just (iii).
	\par
	Summarizing the above two cases \eqref{eq3} and \eqref{eq4}, we have 
	\[
	\limsup_{N\to\infty} \frac{\max_{n\ge0} |A\cap[n+1,n+N]|}{N}
	= \limsup_{N\to\infty} \frac{|A\cap[F(N)+1,F(N)+N]|}{N}
	\le 4\epsilon.
	\]
	Since $0<\epsilon<1/4$ is arbitrary, the set $A$ does not have positive upper Banach density, which is a contradiction.
\end{proof}

\begin{proof}[Proof of Theorem~$\ref{SZ}$]
	Thanks to Proposition~\ref{main5}, it is enough to show that 
	every set $A\subset\mathbb{N}$ with positive upper Banach density satisfies (C). 
	Let $k\geq 3$ be an integer. 
	Lemma~\ref{SZlma1} implies that there exist a function $h=h_\alpha$ (to be independent of $\alpha$) 
	satisfying (H1)--(H4) and a positive number $\delta\le1$ such that
	\[
	\limsup_{x\to \infty} \frac{|A\cap(h(x),h(x+1)]|}{|\mathbb{N}\cap(h(x),h(x+1)]|} = \delta.
	\]
	This equation implies that there exists a strictly increasing sequence $\{x_n\}_{n=1}^\infty\subset\mathbb{R}^+$ which diverges such that 
	every integer $n>0$ satisfies 
	\[
	|A\cap(h(x_n),h(x_n+1)]| \geq \frac{\delta}{2}|\mathbb{N}\cap(h(x_n),h(x_n+1)]|.
	\]
	We can take an integer $n_0>0$ satisfying $h'(x_{n_0}) \ge N(\delta/2,k)$ due to (H4). 
	Since the mean value theorem implies $h(x_n+1) - h(x_n) \ge h'(x_n)$ for all $n\in\mathbb{N}$, 
	every integer $n\ge n_0$ satisfies 
	\[
	h(x_n+1) - h(x_n) \ge h'(x_n) \ge h'(x_{n_0}) \ge N(\delta/2,k).
	\]
	Thus Proposition~\ref{szemeredi} implies that the set $A\cap(h(x_n),h(x_n+1)]$ with $n\ge n_0$ contains an AP of length $k$. 
	Therefore, every $A\subset\mathbb{N}$ with positive upper Banach density satisfies (C).
\end{proof}

\section{Future work}\label{FW}
\begin{question}\label{ques1}
Suppose that a twice differentiable function $f: \mathbb{R}^+\to \mathbb{R}$ satisfies \eqref{SCS1}.
If a strictly increasing sequence $\{a(n)\}_{n=1}^{\infty}\subset \mathbb{N}$ can be written as
\[
	a(n)=f(n)+O(f'(n)),
\]
then does the sequence $\{a(n)\}_{n=1}^{\infty}$ contain arbitrarily long arithmetic progressions?
\end{question}

We do not know the answer to this question, 
but it is affirmative when the coefficients of the error term $O(f'(n))$ are contained in a finite set $\{c_1,c_2,\ldots,c_m\}$. 
See Appendix~\ref{appC}. Hence Question~\ref{ques1} is probably affirmative. 
As a question related to Question~\ref{ques1}, we are interested in how large the error term can be taken. 
In particular, we do not know whether the error term $O(f'(n))$ is best or not.

Next, we remark that the sequence of all primes is a slightly curved sequence: 
the asymptotic expansion \cite{Cipolla} 
\[
	p_n = f(n) + o\Bigl( \frac{n}{\log n} \Bigr)
\]
holds. Here $p_n$ is the $n$-th prime and 
the function $f(x)$ is 
\[
x\{\log x + \log\log x - 1 + (\log\log x - 2)/\log x\},
\]
which satisfies $f''(x) = O(1/x)$. 
Thus $\{p_n\}_{n=1}^\infty$ is a slightly curved sequence with error $o(n/\log n)$. 
If we can improve the error term $O((\log\log n)^{1/c_k})$ in Theorem~\ref{main2} to $o(n/\log n)$, 
the set of all primes contains arbitrarily long arithmetic progressions, 
which was shown by Green and Tao \cite{GreenTao1}. 
If assuming the Riemann hypothesis, the evaluation 
\[
	|p_n - \mathrm{li}^{-1}(n)|\leq \frac{1}{\pi}\sqrt{n}(\log n)^{5/2}
\]
holds \cite[Theorem~6.1]{AriasToulisse}, where $\mathrm{li}^{-1}(x)$ is the inverse function of the logarithmic integral function $\mathrm{li}(x)$. 
Thus $\{p_n\}_{n=1}^\infty$ is a slightly curved sequence with error $O(\sqrt{n}(\log n)^{5/2})$ if assuming the Riemann hypothesis.

\begin{question}\label{ques1.5}
Does the set of all primes satisfy (C)? 
In particular, does the graph of $\{a(p)\}_{p\ \mathrm{prime}}$ contains arbitrarily long arithmetic progressions 
if $\{a(n)\}_{n=1}^\infty$ is a slightly curved sequence with error $O(1)$?
\end{question}

Since the set of all primes does not have positive upper Banach density \cite[Corollary~3.4]{MV}, 
we can not apply Theorem~\ref{SZ} to $\{a(p)\}_{p\ \mathrm{prime}}$. 
Nevertheless, we can answer this question if replacing $p$ with $\mathrm{li}^{-1}(n) + O(1)$.
See Theorem~\ref{appA5}.

\begin{question}\label{ques2}
Is it true that
\[
\sup\Set{a\ge1 | \text{the sequence $\{\lfloor{n^a}\rfloor\}_{n=1}^\infty$ contains arbitrarily long APs}}=2? 
\]
\end{question}

Instead of answering this question, 
we show that the \textit{graph} of $\{\lfloor{n^a}\rfloor\}_{n=1}^\infty$ with $a\ge1$ contains an AP of length $3$ if and only if $1\le a<2$. 
The if part follows from Corollary~\ref{main3}. 
We show the only if part, i.e., 
the graph of $\{\lfloor{n^a}\rfloor\}_{n=1}^\infty$ with $a\ge2$ does not contain any APs of length $3$ by contradiction. 
Suppose that a sequence $\{\lfloor{(e+dj)^a}\rfloor\}_{j=0}^2$ is an AP 
for some two integers $d,e>0$ and some $a\ge2$. 
Then the inequality $|(e+2d)^a + e^a - 2(e+d)^a| < 2$ holds. 
The mean value theorem implies that 
\[
(e+2d)^a + e^a - 2(e+d)^a = d^2 a(a-1)(e + d\theta + d\eta)^{a-2},
\]
where $\theta$ and $\eta$ are real numbers satisfying $\theta,\eta\in(0,1)$. 
Thus it follows that 
\[
2 > |(e+2d)^a + e^a - 2(e+d)^a|
= d^2 a(a-1)(e + d\theta + d\eta)^{a-2} \ge 2,
\]
which is a contradiction.

The above argument implies 
\[
\sup\Set{a\ge1 | \text{the graph of $\{\lfloor{n^a}\rfloor\}_{n=1}^\infty$ contains an AP of length $3$}}=2,
\] 
but we do not achieve the answer to Question~\ref{ques2}.

\appendix
\section{Slightly curved sequences without the assumption $f''(x) = O(1/x^\alpha)$}\label{appA}
As stated in Section~\ref{section1}, Theorem~\ref{main2} requires the assumption $f''(x) = O(1/x^\alpha)$ for some $\alpha>0$. 
This appendix addresses a slightly curved sequence with error $O(1)$ without this assumption. 
First, the following theorem holds: 

\begin{theorem}\label{appA1}
	Let $f$ be a twice differentiable function satisfying $\lim_{x\to\infty} f''(x)=0$ and 
	$\{a(n)\}_{n=1}^\infty\subset\mathbb{N}$ be a strictly increasing sequence satisfying $a(n) = f(n) + O(1)$. 
	Then the graph of $\{a(n)\}_{n=1}^\infty$ contains arbitrarily long APs.
\end{theorem}

Of course, Theorem~\ref{main2} does not completely contain Theorem~\ref{appA1}. 
For example, the function $f(x) = \int_2^x \mathrm{Li}(t)\,dt$ satisfies the assumption in Theorem~\ref{appA1} 
but does not satisfy the assumption in Theorem~\ref{main2} because $f''(x) = 1/\log x$. 
Here $\mathrm{Li}(t)$ is the offset logarithmic integral function, i.e., $\mathrm{Li}(t) := \int_2^t (1/\log s)\,ds$. 
Theorem~\ref{appA1} is derived from Theorem~\ref{appA2} below immediately, which is more exact. 
For example, when length $k$ is given, Theorem~\ref{appA2} implies that the graph of $\{ \fint{c_2 n^2 + c_1 n + c_0} \}_{n=1}^\infty$ 
with a sufficiently small $c_2>0$ contains an AP of length $k$, 
but Theorem~\ref{appA1} does not imply this statement.

\begin{theorem}\label{appA2}
	Let $k\ge3$ and $r\ge1$ be integers, $f$ be a twice differentiable function, 
	$R(n)$ be a bounded function satisfying $M_1\le R(n)\le M_2$ for some two real numbers $M_1$ and $M_2$, and 
	$\{a(n)\}_{n=1}^\infty\subset\mathbb{N}$ be a strictly increasing sequence satisfying $a(n) = f(n) + R(n)$. 
	If the inequality 
	\begin{equation}
		\limsup_{x\to\infty} |f''(x)| < \Bigl( \frac{k-1}{W(r,k)-1} \Bigr)^2 \Bigl( \frac{1}{k-2} - \frac{2(M_2 - M_1)}{r} \Bigr) \label{eqappA4}
	\end{equation}
	holds, then the graph of $\{a(n)\}_{n=1}^\infty$ contains an AP of length $k$.
\end{theorem}

In general, the van der Waerden number $W(r,k)$ is large and thus 
the effect of the factor $W(r,k)$ is larger than that of the other factors of the right-hand side in \eqref{eqappA4}. 
Hence $W(r,k)$ should be small in order to make the right-hand side in \eqref{eqappA4} large. 
To make $W(r,k)$ smallest, we should make $r$ smallest. 
If the right-hand side in \eqref{eqappA4} is positive, 
the smallest integer $r>0$ is 
\begin{equation}
	r = \fint{2(M_2-M_1)(k-2)} + 1 > 2(M_2-M_1)(k-2). \label{eqappA3}
\end{equation}
Thus the integer $r$ with \eqref{eqappA3} makes $W(r,k)$ smallest. 
Let us use \eqref{eqappA3}. 
When $M_2-M_1=1$ and $k=3$, the right-hand side in \eqref{eqappA4} with \eqref{eqappA3} equals $1/(3\cdot13^2)$ because $r=3$ and $W(3,3)=27$ \cite{Chvatal}.

\begin{proof}
	Thanks to \eqref{eqappA4}, 
	there exists an integer $n_0>0$ such that every $x\ge n_0$ satisfies 
	\[
	|f''(x)| < \Bigl( \frac{k-1}{W-1} \Bigr)^2 \Bigl( \frac{1}{k-2} - \frac{2M}{r} \Bigr),
	\]
	where $W = W(r,k)$ and $M = M_2 - M_1$. 
	Then we define the following sets: 
	\begin{gather*}
		I_j = \Bigl[ M_1 + \frac{M}{r}j, M_1 + \frac{M}{r}(j+1) \Bigr]\quad(0\le j\le r-1),\\
		A = \mathbb{N}\cap[n_0, n_0+W-1],\\
		A_j = \set{n\in A | R(n)\in I_j}\quad(0\le j\le r-1).
	\end{gather*}
	It can be easily checked that the interval $[M_1,M_2]$ is the union of all the small intervals $I_j$ and the set $A$ is the union of all the small sets $A_j$. 
	Hence, there exists an integer $0\le q\le r-1$ such that $A_q$ contains an arithmetic progression $P=\{b(j)\}_{j=0}^{k-1}$. 
	Here $b(j)$ is expressed as $b(j) = dj + e$ with two integers $d,e>0$. 
	The inequalities $(k-1)d\le W-1$ and $e\ge n_0$ imply 
	\begin{align}
	\begin{split}
		N_P(f) &= \sum_{j=0}^{k-3} \abs{\Delta^2[f\circ b](j)}
		= \sum_{j=0}^{k-3} d^2\abs{f''(d(j + \theta_j + \eta_j) + e)}\\
		&< \sum_{j=0}^{k-3} \Bigl( \frac{W-1}{k-1} \Bigr)^2 \Bigl( \frac{k-1}{W-1} \Bigr)^2 \Bigl( \frac{1}{k-2} - \frac{2M}{r} \Bigr)
		= 1 - 2(k-2)\frac{M}{r},
	\end{split}\label{eqappA1}
	\end{align}
	where $\theta_j=\theta_j(d,e)$ and $\eta_j=\eta_j(d,e)$ are real numbers satisfying $\theta_j,\eta_j\in(0,1)$. 
	Moreover, the relation $P\subset A_q$ implies 
	\begin{equation}
		N_P(R) = \sum_{j=0}^{k-3} \abs{\Delta^2[R\circ b](j)}
		\le \sum_{j=0}^{k-3} (\abs{\Delta[R\circ b](j+1)} + \abs{\Delta[R\circ b](j)})
		\le 2(k-2)\frac{M}{r}. \label{eqappA2}
	\end{equation}
	Hence the inequality $N_P(a) < 1$ follows from \eqref{eqappA1} and \eqref{eqappA2}. 
	Since $N_P(a)$ is a non-negative integer, the equation $N_P(a)=0$ follows: $a(P)$ is an AP of length $k$. 
	Therefore, the graph of $\{a(n)\}_{n=1}^\infty$ contains an AP of length $k$.
\end{proof}

\begin{remark}
As proved above, van der Waerden's theorem implies Theorem~\ref{appA2} and then Theorem~\ref{appA2} implies Theorem~\ref{appA1}, i.e., 
van der Waerden's theorem implies Theorem~\ref{appA1}. 
Surprisingly, the converse also holds, and thus Theorem~\ref{appA1} is equivalent to van der Waerden's theorem. 
Let us verify the converse. 
Here it is well-known that van der Waerden's theorem is equivalent to the statement that 
if the set of all positive integers is partitioned into finitely many sets 
then at least one set must contain arbitrarily long APs. 
We now show that Theorem~\ref{appA1} implies this statement. 
Suppose that $\mathbb{N}=C_1\cup C_2\cup \cdots \cup C_r$, where $C_i$ and $C_j$ are disjoint for every $1\leq i< j\leq r$. 
We define $R(n)=j$ if $n\in C_j$ and also define $a(n)=rn+R(n)$ for all $n\in\mathbb{N}$. 
Then $\{a(n)\}_{n=1}^\infty$ is a strictly increasing sequence. 
Due to Theorem~\ref{appA1}, the graph of $\{a(n)\}_{n=1}^\infty$ contains arbitrarily long APs. 
Fix an arbitrary integer $k>r$. Then there exists an arithmetic progression $P=\{b(j)\}_{j=0}^{k-1}$ such that 
\begin{equation*}
	0 = N_P(a) = \sum_{j=0}^{k-3} |\Delta^2 [a\circ b ](j)|
	= \sum_{j=0}^{k-3} |\Delta^2 [R\circ b ](j)| = N_P(R).
\end{equation*}
Since the inequalities $1\leq R(n)\leq r$ and $k>r$ hold, the set $R(P)$ is not an AP of length $k$: 
$R(P)=\{j(k)\}$ for some $1\leq j(k)\leq r$, which implies that $C_{j(k)}$ contains an AP of length $k$. 
From the finiteness of the range $1\leq j(k)\leq r$, there exist $1\leq j_0\leq r$ and a sequence $k_1<k_2<\cdots$ such that
\[
C_{j_0}=C_{j(k_1)}=C_{j(k_2)}=\cdots 
\] 
Note that $j_0$ does not depend on $k$. 
Thus $C_{j_0}$ contains arbitrarily long APs. 
From the above discussion, Theorem~\ref{appA1} is equivalent to van der Waerden's theorem.
\end{remark}

We use the proof of Theorem~\ref{appA2} to prove the next theorem. 
As stated in Section~\ref{FW}, we cannot answer Question~\ref{ques1.5}, but the following theorem holds: 

\begin{theorem}\label{appA5}
	Let $\{a(n)\}_{n=1}^\infty$ be a slightly curved sequences with error $O(1)$ and 
	$g$ be a function satisfying the following properties: 
	\begin{itemize}
		\item[{\upshape(G1)}]
		$g$ is a twice differentiable function defined on the interval $(x_0,\infty)$ for some $x_0>0$; 
		\item[{\upshape(G2)}]
		there exists a positive number $x_1\ge x_0$ such that $g$ and $g'$ increase on the interval $(x_1,\infty)$; 
		\item[{\upshape(G3)}]
		$\lim_{x\to\infty} g'(2x)/g(x)^\alpha = 0$ for all $\alpha>0$; 
		\item[{\upshape(G4)}]
		$\lim_{x\to\infty} g''(x) = 0$.
	\end{itemize}
	If $A=\{\tilde{a}(n)\}_{n=n_0}^\infty$ with $n_0>x_1$ is a strictly increasing sequence satisfying $\tilde{a}(n)=g(n)+O(1)$, 
	then the graph of $\{a(n)\}_{n\in A}$ contains arbitrarily long APs.
\end{theorem}
\begin{proof}
	Thanks to Proposition~\ref{main5}, it is enough to show that 
	the set $A=\{\tilde{a}(n)\}_{n=n_0}^\infty$ satisfies (C). 
	Let $k\ge3$ be an integer and $\alpha$ be a positive number. 
	We choose the function $h_\alpha$ as $h_\alpha(x)=x^\beta$ with $\beta=1+\alpha/2$, 
	which satisfies (H1)--(H4). Hence all we need is to show that 
	there exists a strictly increasing sequence $\{x_n\}_{n=1}^\infty\subset\mathbb{R}^+$ which diverges such that 
	for every $n>0$ the set $A\cap(h_\alpha(x_n),h_\alpha(x_n+1)]$ contains an AP of length $k$.
	\par
	Take two real numbers $M_1$ and $M_2$ satisfying $M_1 \le \tilde{a}(n) - g(n) \le M_2$ and the integer $r$ with \eqref{eqappA3}, 
	and let $W=W(r,k)$ and $M=M_2-M_1$. 
	Since $g$ satisfies (G4), Theorem~\ref{appA2} can be applied to $A=\{\tilde{a}(n)\}_{n=n_0}^\infty$. 
	More precisely, the proof of Theorem~\ref{appA2} implies that 
	there exists an integer $n_1\ge n_0$ such that for every $n\ge n_1$ 
	the set $A\cap[\tilde{a}(n),\tilde{a}(n+W-1)]$ contains an AP of length $k$. 
	Thus (C) can be reduced to the inclusion relation 
	\begin{equation}
		A\cap(h_\alpha(x_n),h_\alpha(x_n+1)] \supset A\cap[\tilde{a}(n),\tilde{a}(n+W-1)]. \label{eqY2}
	\end{equation}
	Let us show \eqref{eqY2}. Put $x_n = (g(n)+M_1-1)^{1/\beta} < \tilde{a}(n)^{1/\beta}$. 
	Then $h_\alpha(x_n) = g(n)+M_1-1 < \tilde{a}(n)$. 
	The mean value theorem implies 
	\begin{align}
		\tilde{a}(n+W-1) &\le g(n+W-1) + M_2
		\le g(n) + (W-1)g'(n+W-1) + M_2\nonumber\\
		&= h_\alpha(x_n) - (M_1-1) + (W-1)g'(n+W-1) + M_2\nonumber\\
		&= h_\alpha(x_n) + (W-1)g'(n+W-1) + M + 1, \label{eqY5}\\
		h_\alpha(x_n+1) &\ge h_\alpha(x_n) + h'_\alpha(x)
		= h_\alpha(x_n) + \beta x_n^{\beta-1} \label{eqY6}
	\end{align}
	due to (G2). Using (G2), (G3), and $\lim_{x\to\infty} g(x)=\infty$ 
	(this limit follows from the fact that $\{\tilde{a}(n)\}_{n_0}^\infty$ is a strictly increasing sequence), 
	we have 
	\[
	\lim_{x\to\infty} \frac{(W-1)g'(x+W-1) + M + 1}{\beta (g(x) + M_1 - 1)^{1-1/\beta}}
	\le \lim_{x\to\infty} \frac{(W-1)g'(2x) + M + 1}{(g(x)/2)^{1-1/\beta}} = 0,
	\]
	whence $(W-1)g'(n+W-1) + M + 1 = o(\beta x_n^{\beta-1})$. 
	Thus there exists an integer $n_2\ge n_1$ such that every $n\ge n_2$ satisfies 
	\begin{equation}
		(W-1)g'(n+W-1) + M + 1 \le \beta x_n^{\beta-1}. \label{eqY7}
	\end{equation}
	The equations \eqref{eqY5}, \eqref{eqY6}, and \eqref{eqY7} yield 
	$h_\alpha(x_n+1)\ge\tilde{a}(n+W-1)$ for every $n\ge n_2$. 
	Therefore, the inclusion relation \eqref{eqY2} holds for every $n\ge n_2$.
\end{proof}

\begin{remark}
	The function $g(x)=\mathrm{li}^{-1}(x)$ satisfies (G1)--(G4). (The definition of $\mathrm{li}^{-1}(x)$ is in Section~\ref{FW}.) 
	Moreover, when $\tilde{a}(n)=\mathrm{li}^{-1}(n)+O(1)$, 
	the sequence $A=\{\tilde{a}(n)\}_{n=n_0}^\infty$ for some $n_0$ is a strictly increasing sequence. 
	Therefore, if $\{a(n)\}_{n=1}^\infty$ is a slightly curved sequence with error $O(1)$, 
	the graph of $\{a(n)\}_{n\in A}$ contains arbitrarily long APs.
\end{remark}

\section{Result similar to Theorem~\ref{appA2}}\label{appB}
This appendix states a result similar to Theorem~\ref{appA2}, which is not completely contained by Theorem~\ref{appA2}. 
Indeed, Theorem~\ref{appB3} below implies that the graph of 
$\{ \fint{c_2 n^2 + c_1 n + c_0} \}_{n=1}^\infty$ with $0<c_2<1/18$ contains an AP of length $4$, 
which is better evaluation than that of Theorem~\ref{appA2}.

\begin{theorem}\label{appB2}
	If a function $f: \mathbb{N}\to\mathbb{R}^+$ satisfies $\Delta f > 0$, $\Delta^2 f \ge 0$, and 
	\begin{equation}
		\sup_{x>0} N(x)^2/x = \infty, \label{eq:sup}
	\end{equation}
	then the graph of $\{ \fint{f(n)} \}_{n=1}^\infty$ contains an AP of length $4$. 
	Here $N(x)$ denotes the number of positive integers satisfying $f(n) \le x$, i.e., $N(x) := \abs{\set{n\in\mathbb{N} | f(n) \le x}}$.
\end{theorem}

\begin{remark}
Theorem~\ref{appB2} does not assume that the sequence $\{ \fint{f(n)} \}_{n=1}^\infty$ strictly increases, 
but when the function $f$ satisfies $\Delta f(n_0)\ge1$ for some $n_0>0$, 
the sequence $\{ \fint{f(n)} \}_{n=n_0}^\infty$ strictly increases due to the assumption $\Delta^2 f \ge 0$.
\end{remark}

Let $\Sigma$ be a nonempty finite set of non-negative integers. 
In order to show Theorem~\ref{appB2}, we address \textit{words} over the \textit{alphabet} $\Sigma$. 
Here we allow words to continue infinitely on the right side such as $000\cdots$, and to be the empty word. 
The length of a word $w$ (i.e., the number of all \textit{letters} of $w$) is denoted by $|w|$ and the sum of all letters of a word $w$ is denoted by $\sum w$. 
Then we focus on the following condition (C${}_k$) for a word $w$: 
\textit{there exist $k+1$ finite length words $w_0,\ldots,w_k$ and a word $w_{k+1}$ such that} 
\[
\text{$w = w_0 \cdots w_k w_{k+1}$, $|w_1|=\cdots=|w_k|>0$, and $\sum w_1=\cdots=\sum w_k$.}
\]
We also focus on the following statement: 
\textit{there exists an integer $n>0$ such that every word $w$ with $|w|\ge n$ satisfies the condition (C${}_k$). }
We call it the statement (S${}_k$). 
The statement (S${}_k$) depends on the alphabet $\Sigma$. 
Table~\ref{T1} summarizes integers $n$ in the statements (S${}_2$) and (S${}_3$) for several alphabets $\Sigma$.
The proof of Theorem~\ref{appB2} uses the statement (S${}_3$). 
The statement (S${}_3$) for $\Sigma=\{ 0,1 \}$ can be checked by taking $n=10$. 
Cassaigne et al.\ \cite{Cassaigne} proved that the statement (S${}_3$) did not hold for $\Sigma=\{ 0,1,3,4 \}$. 
For details, see \cite{Cassaigne}.

\begin{table}[t]
	\centering
	\caption{Integers $n$ in the statement (S${}_k$) for several alphabets $\Sigma$. 
	The case when $k=2$ and $\Sigma=\{0,1,2,3\}$ was computed by a personal computer. 
	If the statement (S${}_2$) for $\Sigma=\{0,1,2,3,4\}$ holds, then $n>152$ by computing.
	If the statement (S${}_3$) for $\Sigma=\{0,1,2\}$ holds, then $n>1288$ \cite{Cassaigne}.
	}\label{T1}
	\begin{tabular}{|c||c|c|c|c|}\hline
		   &$\Sigma=\{0,1\}$&$\Sigma=\{0,1,2\}$&$\Sigma=\{0,1,2,3\}$&$\Sigma=\{0,1,3,4\}$\\
		\hhline{|=#=|=|=|=|}
		$k=2$&$n=4$&$n=8$&$n=51$&   \\ \hline
		$k=3$&$n=10$&   &   &(S${}_3$) does not hold \cite{Cassaigne}.\\ \hline
	\end{tabular}
\end{table}

\begin{proof}[Proof of Theorem~$\ref{appB2}$]
	We show that for every integer $k>0$ there exists an integer $n_0>0$ such that 
	the function $\fint{\Delta f(n)}$ takes a constant value for every $n_0\le n\le n_0+k$, by contradiction. 
	Suppose that our assertion does not hold, namely, 
	for some integer $k>0$, there exists no positive integer $n_0$ such that 
	the function $\fint{\Delta f(n)}$ takes a constant value for every $n_0\le n\le n_0+k$. 
	Putting the function $\lambda(x) := \cint{N(x)/k}-1$, we find $N(x) = k N(x)/k > k\lambda(x)$. 
	Hence the inequality $N(x)\ge k\lambda(x)+1$ holds. 
	Then the assumption $\Delta^2 f(n)\ge0$ yields $\fint{\Delta f(n+1)} \ge \fint{\Delta f(n)}$. 
	Using this inequality and the assumption of the proof by contradiction to obtain the following (i), we have 
	\begin{align*}
		x &> f(k\lambda(x)+1) - f(1) = \sum_{j=1}^{k\lambda(x)} \Delta f(j)
		\ge \sum_{j=1}^{k\lambda(x)} \fint{\Delta f(j)}
		\overset{\text{(i)}}{\ge} k\sum_{j=0}^{\lambda(x)-1} j\\
		&= \frac{k}{2}\lambda(x)(\lambda(x)-1)
		\ge \frac{k}{2}\Bigl( \frac{N(x)}{k} - 1 \Bigr)\Bigl( \frac{N(x)}{k} - 2 \Bigr),
	\end{align*}
	that is, $(N(x)-k)(N(x)-2k) < 2kx$. 
	This inequality implies that $N(x)^2/x$ is upper bounded by a constant, 
	which contradicts the assumption \eqref{eq:sup}.
	\par
	Putting $k=9$, we can take two integers $n_0>0$ and $c\ge0$ satisfying $\fint{\Delta f(n)}=c$ for all $n_0\le n\le n_0+9$. 
	Thus $\Delta \fint{f(n)} - c\in\{ 0,1 \}$ for all $n_0\le n\le n_0+9$. 
	Since the statement (S${}_3$) holds for $\Sigma=\{ 0,1 \}$, 
	there exist two integers $d,e>0$ such that 
	\[
	\sum_{j=e}^{e+d-1} (\Delta \fint{f(j)} - c) = \sum_{j=e+d}^{e+2d-1} (\Delta \fint{f(j)} - c) = \sum_{j=e+2d}^{e+3d-1} (\Delta \fint{f(j)} - c).
	\]
	Thus the equation 
	\[
	\fint{f(e+d)} - \fint{f(e)} = \fint{f(e+2d)} - \fint{f(e+d)} = \fint{f(e+3d)} - \fint{f(e+2d)}
	\]
	holds: $\{ \fint{f(e+dj)} \}_{j=0}^3$ is an AP of length $4$. 
	That is, the graph of $\{ \fint{f(n)} \}_{n=0}^\infty$ contains an AP of length $4$.
\end{proof}

\begin{theorem}\label{appB3}
	If a function $f: \mathbb{N}\to\mathbb{R}^+$ satisfies $\Delta f > 0$, $\Delta^2 f \ge 0$, and 
	\begin{equation}
		\limsup_{x\to\infty} N(x)^2/x > 18, \label{eq:limsup}
	\end{equation}
	then the graph of $\{ \fint{f(n)} \}_{n=1}^\infty$ contains an AP of length $4$.
\end{theorem}
\begin{proof}
	If the function $N(x)^2/x$ is not bounded, our assertion follows from Theorem~\ref{appB2}. 
	Hence we assume that the function $N(x)^2/x$ is bounded. 
	We show that there exists an integer $n_0>0$ such that 
	the function $\fint{\Delta f(n)}$ takes a constant value for every $n_0\le n\le n_0+9$, by contradiction. 
	In the same way as the proof of Theorem~\ref{appB2}, 
	the inequality $(N(x)-9)(N(x)-18) < 18x$ holds. 
	Moreover, since the function $N(x)^2/x$ is bounded, the limit 
	\[
	\frac{N(x)}{x} = \frac{N(x)}{x^{1/2}}\frac{1}{x^{1/2}} \xrightarrow{x\to\infty} 0
	\]
	holds. Thus the inequality $\limsup_{x\to\infty} N(x)^2/x \le 18$ follows, which contradicts the assumption \eqref{eq:limsup}. 
	Therefore, there exists two integers $n_0$ and $c$ such that $\fint{\Delta f(n)}=c$ for every $n_0\le n\le n_0+9$. 
	The remaining can be shown in the same way as the proof of Theorem~\ref{appB2}.
\end{proof}

\begin{remark}
Theorem~\ref{appB3} implies that the graph of 
$\{ \fint{c_2 n^2 + c_1 n + c_0} \}_{n=1}^\infty$ with $0<c_2<1/18$ contains an AP of length $4$. 
However, 
Theorem~\ref{appA2} only implies that the same graph with $0<c_2<1/(6\cdot13^2)$ contains an AP of length $3$.
\end{remark}

The following propositions provide alternative conditions of \eqref{eq:sup} and \eqref{eq:limsup}.

\begin{proposition}
	If a function $f: \mathbb{N}\to\mathbb{R}^+$ satisfies $\Delta f > 0$ and $\Delta^2 f \ge 0$, 
	then for every $c>0$ the following three conditions are equivalent: 
	\[
	\text{{\upshape(i)} $\limsup_{x\to\infty} N(x)^2/x > c$, 
	{\upshape(ii)} $\limsup_{n\to\infty} n^2/f(n) > c$, and 
	{\upshape(iii)} $\liminf_{n\to\infty} f(n)/n^2 < 1/c$.}
	\]
\end{proposition}
\begin{proof}
	First, note that the limit $f(n) \ge f(1) + (n-1)\Delta f(1) \to \infty$ holds as $n\to\infty$. 
	The implication (ii)$\Rightarrow$(i) follows from the definition of $N(x)$. 
	Next, we show the implication (i)$\Rightarrow$(ii) by contradiction. 
	Suppose that the condition (i) holds and the condition (ii) does not hold, i.e., 
	$\limsup_{x\to\infty} N(x)^2/x > c$ and $\limsup_{n\to\infty} n^2/f(n) \le c$. 
	The assumption $\Delta f > 0$ and the limit $f(n)\to\infty$ imply that 
	for every $x>0$ we can take only one positive integer $n=n(x)$ satisfying $f(n) \le x < f(n+1)$. 
	Hence every $x>0$ satisfies 
	\begin{equation}
		N(x)^2/x < (n+1)^2/f(n) = n^2/f(n) + 2n/f(n) + 1/f(n),\label{eqappB1}
	\end{equation}
	where $n=n(x)$. 
	Now, the assumption $\limsup_{n\to\infty} n^2/f(n) \le c$ yields $\limsup_{n\to\infty} n/f(n) = 0$ and $\limsup_{n\to\infty} 1/f(n) = 0$. 
	Thus, noting $n=n(x)\to\infty$ as $x\to\infty$ and taking the limit in \eqref{eqappB1} as $x\to\infty$, 
	we have 
	\[
	c < \limsup_{x\to\infty} N(x)^2/x \le \limsup_{n\to\infty} n^2/f(n) \le c,
	\]
	which is a contradiction. 
	Therefore, the implication (i)$\Rightarrow$(ii) holds. 
	The remaining, i.e., the equivalence (ii)$\Leftrightarrow$(iii) is trivial.
\end{proof}

\begin{proposition}
	If a function $f: \mathbb{N}\to\mathbb{R}^+$ satisfies $\Delta f > 0$ and $\sum_{n=1}^\infty 1/f(n)^s = \infty$ for some $s>1/2$, 
	then the condition $\sup_{x>0} N(x)^2/x = \infty$ holds.
\end{proposition}
\begin{proof}
	Let $\zeta$ be the Riemann zeta function. 
	We show our assertion by contradiction. 
	Suppose that the function $N(x)^2/x$ is bounded. 
	Then the inequality $N(x)^2/x \le M$ holds for some $M>0$. 
	Thus the inequality $n^2/f(n) \le M$ also holds and every $s>1/2$ satisfies 
	$\sum_{n=1}^\infty 1/f(n)^s \le M\zeta(2s) < \infty$, which is a contradiction.
\end{proof}

\section{Partial answer to Question~\ref{ques1}}\label{appC}
This section proves that Question~\ref{ques1} is affirmative 
when the coefficients of the error term $O(f'(n))$ are contained in a finite set $\{c_1,c_2,\ldots,c_m\}$.

\begin{theorem}\label{appC1}
	Let $f: \mathbb{R}^+\to \mathbb{R}$ be a twice differentiable function satisfying \eqref{SCS1}, and 
	$\tilde{R}: \mathbb{N}\to\{c_1,c_2,\ldots,c_m\}$ be a bounded function. 
	If a strictly increasing sequence $\{a(n)\}_{n=1}^{\infty}\subset \mathbb{N}$ can be written as
	\[
	a(n) = f(n) + \tilde{R}(n)f'(n) + O(1),
	\]
	then the sequence $\{a(n)\}_{n=1}^{\infty}$ contains arbitrarily long APs.
\end{theorem}
\begin{proof}
	Fix an integer $k\ge3$. First, the assumption \eqref{SCS1} implies 
	\begin{equation}
		a(n) = f(n) + \tilde{R}(n)f'(n) + O(1)
		= f(n+\tilde{R}(n)) + O(1). \label{eqappC1}
	\end{equation}
	We define $\tilde{a}(n)=n+\tilde{R}(n)$, $W=W(m,W(r,k))$, and the following sets: 
	\begin{gather*}
		\tilde{A} = \mathbb{N}\cap[\lfloor W^{3/\alpha} \rfloor, \lfloor W^{3/\alpha} \rfloor+W-1],\\
		\tilde{A}_j = \set{n\in \tilde{A} | \tilde{R}(n)=c_j}\quad(1\le j\le m).
	\end{gather*}
	Since the union of all the small sets $\tilde{A}_j$ is $\tilde{A}$, 
	a small set $\tilde{A}_{\tilde{q}}$ contains an arithmetic progression $\tilde{P}$ of length $W(r,k)$. 
	Thus it follows that 
	\[
	N_{\tilde{P}}(\tilde{a}) \le N_{\tilde{P}}(\id_{\mathbb{N}}) + N_{\tilde{P}}(\tilde{R}) = 0.
	\]
	In other words, the sequence $\tilde{a}(\tilde{P})$ (of real numbers) is also an AP of length $W(r,k)$.
	\par 
	Noting that the sequence $\tilde{a}(\tilde{P})$ is a strictly increasing sequence, 
	we define the function $R: \tilde{a}(\tilde{P})\to\mathbb{R}$ as 
	\begin{equation}
		R(\tilde{a}(n)) = a(n) - f(\tilde{a}(n))\quad(n\in\tilde{P}). \label{eqappC2}
	\end{equation}
	Then the function $R$ is bounded. 
	Thanks to \eqref{eqappC1}, we can take two real numbers $M_1$ and $M_2$ satisfying $M_1\le R(x)\le M_2$ 
	such that $M_1$ and $M_2$ are independent of $r$. We define $M=M_2-M_1$ and the following sets: 
	\begin{gather*}
		I_j = \Bigl[ M_1 + \frac{M}{r}j, M_1 + \frac{M}{r}(j+1) \Bigr]\quad(0\le j\le r-1),\\
		A_j = \set{x\in \tilde{a}(\tilde{P}) | R(n)\in I_j}\quad(0\le j\le r-1).
	\end{gather*}
	Since the union of all the small sets $A_j$ is $\tilde{a}(\tilde{P})$, 
	a small set $A_q$ contains an arithmetic progression $P=\{b(j)\}_{j=0}^{k-1}$ of length $k$. 
	Here $b(j)$ is expressed as $b(j) = dj + e$ with two integers $d,e>0$. 
	Moreover, the inequalities 
	\begin{gather*}
		d \le \tilde{a}(\lfloor W^{3/\alpha} \rfloor+W-1) - \tilde{a}(\lfloor W^{3/\alpha} \rfloor)
		\le W-1+M,\\
		e \ge \tilde{a}(\lfloor W^{3/\alpha} \rfloor) \ge \lfloor W^{3/\alpha} \rfloor + M_1
	\end{gather*}
	hold. These inequalities imply that 
	\begin{gather*}
		N_P(R) \le 2(k-2)\frac{M}{r} = o(1),\\
		N_P(f) = \sum_{j=0}^{k-3} d^2\abs{f''(d(j + \theta_j + \eta_j) + e)}
		= O\Bigl( \frac{(W-1+M)^2}{(\lfloor W^{3/\alpha} \rfloor + M_1)^\alpha} \Bigr)
		= O(1/W) = o(1)
	\end{gather*}
	as $r\to\infty$, where $\theta_j=\theta_j(d,e)$ and $\eta_j=\eta_j(d,e)$ are real numbers satisfying $\theta_j,\eta_j\in(0,1)$. 
	Thus every sufficiently large $r$ satisfies $N_P(f) + N_P(R) < 1$.
	\par
	Recall \eqref{eqappC2}. The domains of $f$ and $R$ are the same, but the domain of $a$ is different from them. 
	Hence we must take a set like the inverse image of $P$ under $\tilde{a}$. 
	Now, the restricted function $\tilde{a}|\tilde{P}$ is injective. 
	When we denote by $P'$ the inverse image of $P\subset\tilde{a}(\tilde{P})$ under this restricted function, 
	the sequence $P'$ is also an AP of length $k$. Then the inequality $N_{P'}(a)\le N_P(f) + N_P(R) < 1$ holds. 
	Since $N_{P'}(a)$ is a non-negative integer, the equation $N_{P'}(a)=0$ follows: $a(P')$ is an AP of length $k$.
\end{proof}

\end{document}